\documentclass[final,leqno,onefignum,onetabnum]{siamltex1213}
\usepackage{caption2}
\usepackage{pdfsync}
\usepackage{amssymb}
\usepackage{amsmath}

\usepackage{amsthm}
\usepackage{graphicx}
\usepackage{empheq}
\everymath{\displaystyle}
\usepackage{color}
\usepackage[usenames,dvipsnames,svgnames,table]{xcolor}

\newcommand{\tr}{^{\! \mathsf{T}}}

\usepackage{hyperref}
\hypersetup{
     colorlinks   = true,
      linkcolor=red,
      urlcolor=blue,
     citecolor    = blue
}

\newtheorem{remark}{Remark}[section]
 \newtheorem{thm}{Theorem}[section]
 \newtheorem{cor}{Corollary}[section]
 \newtheorem{lem}{Lemma}[section]
 \newtheorem{prop}{Proposition}[section]
 \theoremstyle{definition}
 \newtheorem{defn}{Definition}[section]
 \newtheorem{rem}{Remark}[section]
 \numberwithin{equation}{section}
  
\newcommand{\R}{\mathbb{R}}

\title{Asymptotic stability of a Korteweg-de Vries \\equation with a two-dimensional  center manifold}

\author{Shuxia Tang\footnotemark[1]\ 
\and Jixun Chu\footnotemark[2]\ 
\and Peipei Shang\footnotemark[3]
\and Jean-Michel Coron\footnotemark[4]\ 
}

\begin{document}

\maketitle

\renewcommand{\thefootnote}{\fnsymbol{footnote}}

\footnotetext[1]{S. Tang is with Department
of Mechanical \& Aerospace Engineering, University of California, San Diego,
La Jolla, 92093, CA, USA,
and also with
Universit\'{e} Pierre et Marie Curie-Paris 6,
UMR 7598 Laboratoire Jacques-Louis Lions, 75005 Paris, France. (\email{sht015@ucsd.edu}). S. Tang was supported by ERC advanced grant 266907 (CPDENL) of the 7th Research Framework Programme (FP7).}
\footnotetext[2]{J. Chu is with Department of Applied
Mathematics, School of Mathematics and Physics, University of Science and
Technology Beijing,
Beijing 100083, China. (\email{chujixun@126.com}). J. Chu was partially supported by ERC advanced grant 266907 (CPDENL) of the 7th Research
Framework Programme (FP7), National Natural Science
Foundation of China  (No.11401021) and Doctoral Program of Higher Education
(No.20130006120011).}

\footnotetext[3]{P. Shang is with Department of
Mathematics, Tongji
University, Shanghai 200092, China. (\email{peipeishang@hotmail.com}). P. Shang was partially supported by ERC advanced grant 266907(CPDENL), National Natural Science Foundation of China (No.11301387) and Doctoral Program of Higher Education (No. 20130072120008)}
\footnotetext[4]{J.-M. Coron is with Universit\'{e} Pierre et Marie Curie-Paris 6, UMR 7598
Laboratoire Jacques-Louis Lions, 75005 Paris, France. (\email{coron@ann.jussieu.fr}). J.-M. Coron was supported by ERC advanced grant 266907 (CPDENL) of the 7th Research
Framework Programme (FP7).}

\begin{abstract}
Local asymptotic stability analysis is conducted for an initial-boundary-value problem of a  Korteweg-de Vries
equation posed on a finite interval $\left[0, 2\pi \sqrt{7/3}\right]$.
The equation comes with
 a Dirichlet
boundary condition at the left end-point and  both of the Dirichlet and Neumann
homogeneous boundary conditions at the right end-point.  It is known that the
associated linearized  equation around the origin is
not asymptotically stable. In this paper,  the nonlinear
Korteweg-de Vries equation is proved to be
locally asymptotically stable around the origin through the center
manifold method.
 In particular, the existence of a  two-dimensional local  center manifold is presented, which is locally exponentially attractive.
 By analyzing the Korteweg-de Vries equation restricted on the local center manifold,
a polynomial decay rate of the solution is obtained.

\end{abstract}



\pagestyle{myheadings}
\thispagestyle{plain}

\section{Introduction}\

The Korteweg-de Vries (KdV) equation
\begin{align}
y_{t}+y_{x}+yy_{x}+y_{xxx}=0\label{KDVcontrol0}
\end{align}
 was
first derived by Boussinesq in \cite[Equation (283 bis)]{B1877} and
 by Korteweg and de Vries in \cite{KV1895}, for describing the propagation
of small amplitude long water waves in a uniform channel. This equation is now commonly used to model unidirectional
propagation of small amplitude long waves in nonlinear dispersive systems.
An excellent reference to help understand both physical motivation and deduction
of the KdV equation is the book by Whitham \cite{W1974}.

 Rosier studied in \cite{R1997} the following nonlinear Neumann boundary control problem for
 the KdV equation with homogeneous Dirichlet boundary conditions,
 posed on a finite spatial interval:
 \begin{equation}
 \left\{
 \begin{array}
 [c]{l}
 y_{t}+y_{x}+yy_{x}+y_{xxx}=0,~t\in(0,\infty),~x\in (0,L),\\
 y(t,0)=y(t,L)=0,~
 y_{x}(t,L)=u(t),~t\in(0,\infty),\\
 y(0,x)=y_{0}(x) , ~x\in(0,L),
 \end{array}
 \right.  \label{KDV
 control}
 \end{equation}
 where $L>0$, the state is $y(t,\cdot):[0,L]\rightarrow\mathbb{R}$, and
  $u(t)\in\mathbb{R}$  denotes the  controller. The equation
 comes with one boundary condition at the left
 end-point and two boundary conditions at the right end-point. He  first
 considered the first
 order power series expansion
 of $(y,u)$ around the origin, which gives the following corresponding linearized
 control
 system
 \begin{equation}
 \left\{
 \begin{array}
 [c]{l}
 y_{t}+y_{x}+y_{xxx}=0,~t\in(0,\infty),~x\in (0,L),\\
 y(t,0)=y(t,L)=0,~y_{x}(t,L)=u(t),~t\in(0,\infty),\\
 y(0,x)=y_{0}(x),~x\in(0,L) .
 \end{array}
 \right.  \label{linearized00}
 \end{equation}
 By means of multiplier technique and the Hilbert Uniqueness Method
 (HUM) \cite{L1988}, he proved that (\ref{linearized00})
 is exactly controllable if and only if the length
 of the spatial domain is not critical,
 i.e.,
 $L\notin\mathcal{N}$, where $\mathcal{N}$ denotes following set of critical
 lengths
\begin{align}
 \mathcal{N}:=\left\{  2\pi\sqrt{\frac{j^{2}+l^{2}+jl}{3}};\;j,l\in
 \mathbb{N}^{\ast}\right\}.
 \label{defN}
\end{align}
 Then, by employing the  Banach fixed point theorem, he derived
 that the nonlinear  KdV control system \eqref{KDV
 control}
 is locally exactly controllable around $0$ provided that $L\notin\mathcal{N}$.
 In the cases with critical lengths
 $L\in\mathcal{N}$,  Rosier demonstrated  in \cite{R1997} that there exists
 a finite dimensional subspace $M$ of $L^2(0, L)$ which
 is unreachable for the linear system  \eqref{linearized00} when starting
 from the origin. In \cite{CC2004}, Coron
 and Cr\'{e}peau  treated a critical case of $L=2k\pi$ (i.e., taking
 $j=l=k$ in $\mathcal{N}$), where $k$ is a positive integer
 such that (see, \cite[Theorem 8.1 and
 Remark 8.2]{2007-Coron-book})
 \begin{equation}\label{condition-sur-pour-espapace-dim1}
 \left(  j^{2}+l^{2}+ jl=3k^{2}\text{ and }j,l\in\mathbb{N}^*\right)
  \Rightarrow\left(  j=l=k\right)  .
 \end{equation}Here, the uncontrollable subspace $M$ for the linear system
 \eqref{linearized00} is one-dimensional. However, through a third order power
 series expansion of the solution, they showed that the nonlinear term $yy_{x}$
 always allows to ``go'' in small-time into the  two directions missed by the linearized control
 system
 \eqref{linearized00}, and then, using a fixed point theorem, they deduced the small-time
 local exact
 controllability around the origin of the nonlinear control system
 \eqref{KDV control}.
 In
 \cite{C2007}, Cerpa studied  the critical case of $L\in \mathcal{N}'$, where
 \begin{align}
 \mathcal{N}':=\bigg\{&2\pi\sqrt{\frac{j^{2}+l^{2}+jl}{3}};\;j,l\in
 \mathbb{N}^{\ast} ~ {\text{satisfying}}~ j>l ~{\text{and}}~\nonumber\\
&~~~~~~~~j^2+jl+l^2\neq m^2+mn+n^2, \forall m,n\in \mathbb{N}^*\backslash \{j\}\bigg\}.
 ~~~~
 \label{N'1}
 \end{align}
 In this case, the uncontrollable subspace $M$  for the  linear system
 \eqref{linearized00} is of dimension 2, and the author used a second order
 expansion of the solution to the  nonlinear control system
 \eqref{KDV control} to prove the local exact
 controllability in large time around the origin of the nonlinear control system
 \eqref{KDV control} (the local controllability in small time for this length $L$ is still an open problem).
 Furthermore,  Cerpa and Cr\'{e}peau considered in \cite{2007-Coron-book} the cases when the dimension
 of $M$ for the
  linear system
 \eqref{linearized00} is higher than 2.
 They implemented a second order expansion of the solution to
 \eqref{KDV control}  for the critical lengths $L\neq2k\pi$ for any $k\in \mathbb{N}^*$,
 and implemented an expansion to the third order if $L=2k\pi$ for some  $k\in
 \mathbb{N}^*$. They showed that the nonlinear term $yy_{x}$
 always allows to ``go'' into all the  directions missed by the linearized control
 system
 \eqref{linearized00}   and then proved the local exact controllability in large time around the origin of the nonlinear control system
 \eqref{KDV control}.

Consider the case when there is no control, i.e., $u=0,$ in \eqref{KDV
 control},
which gives the following
initial-boundary-value KdV
problem 
posed on a finite interval $[0,L]$:
\begin{equation}
\left\{
\begin{array}
[c]{l}
y_{t}+y_x+y_{xxx}+yy_{x}=0,~t\in(0,\infty),~x\in (0,L),\\
y(t,0)=y(t,L)=0,~y_{x}(t,L)=0,~t\in(0,\infty),\\
y(0,x)=y_{0}(x),~x\in(0,L),
\end{array}
\right.  \label{original}
\end{equation}
where the boundary conditions are homogeneous. For  the Lyapunov function
\begin{align}
\label{defE}
E(t)=\frac{1}{2}\Vert y(t,\cdot)\Vert_{L^2(0,L)}^2=\frac{1}{2}\int_0^L  y^2(t,x) dx,
\end{align}
we have
\begin{align}
\dot E (t)&=-\int_0^L y(y_{x}+yy_{x}+y_{xxx})dx=\int_0^L y_x y_{xx}dx=-\frac{1}{2}
y_x^2(t,0)\leq 0.\label{globallstabb}
\end{align}
Thus, $0\in L^2(0,L)$ is stable (see \hyperref[P1]{($\mathcal{P}_1$)} below for the definition of stable) for the KdV equation (\ref{original}).  Moreover, it has been proved in \cite{MVZ2002} that,  if $L\notin\mathcal{N}$, then $0$ is
exponentially stable for the corresponding
linearized equation around the origin
\begin{equation}
\left\{
\begin{array}
[c]{l}
y_{t}+y_{x}+y_{xxx}=0,~t\in(0,\infty),~x\in (0,L),\\
y(t,0)=y(t,L)=0,~y_{x}(t,L)=0,~t\in(0,\infty),~\\
y(0,x)=y_{0}(x),~x\in(0,L),
\end{array}
\right.  \label{linearized0}
\end{equation}
which gives the local  asymptotic stability around the origin for the nonlinear equation \eqref{original}. However, when $L\in \mathcal{N}$, Rosier pointed out in \cite{R1997} that  the  equation \eqref{linearized0}
is not asymptotically stable. Inspired by the fact that the nonlinear
term $yy_{x}$ introduces the local exact controllability around the origin into the KdV control system
 \eqref{KDV
 control} with $L\in\mathcal{N}$, we would like to  discuss whether   the nonlinear term $yy_{x}$ could introduce local asymptotic stability around the origin for
\eqref{original}.

This paper is devoted to investigating the local asymptotic stability  of $0\in L^2(0,L)$ for
\eqref{original}
with the  critical length
\begin{equation}
\label{valueL}
L=2\pi \sqrt{\frac{7}{3}},
\end{equation}
corresponding to $j=1$ and $l=2$ in \eqref{defN}. Let us recall that this local asymptotic stability means that the following two properties are satisfied.
\begin{itemize}
\item[\label{P1}($\mathcal{P}_1$)]  Stability:
for every $\varepsilon>0$, there exists $\eta=\eta(\varepsilon)>0$ such that, if $\Vert
y_{0}\Vert_{L^{2}(0,L)}<\eta$,  then
\begin{align}
\Vert y(t,\cdot)\Vert_{L^{2}(0,L)}<\varepsilon,\quad\forall t\geq0.
\end{align}

\item[\label{P2}($\mathcal{P}_2$)] (Local) attractivity: there exists $\varepsilon_0>0$ such that, if $\Vert y_{0}\Vert
_{L^{2}(0,L)}<\varepsilon_0$,  then
\begin{align}
\lim_{t\rightarrow+\infty}\Vert y(t,\cdot)\Vert_{L^{2}(0,L)}=0.
\end{align}
\end{itemize}
As mentioned above, the stability property \hyperref[P1]{($\mathcal{P}_1$)} is implied by \eqref{globallstabb}.  Our main concern is thus the local attractivity property \hyperref[P2]{($\mathcal{P}_2$)}. We prove the following theorem, where
the precise definition of a solution to \eqref{original} is given in Definition \ref{def_kdv_sol} and  the precise definition of the finite dimensional vector space $M\subset L^2(0,L)$ when $L=2\pi \sqrt{7/3}$ is given in \eqref{spacedec}.
\begin{thm}\label{main result}
Consider the {\text {KdV}} equation
(\ref{original}) with $L=2\pi \sqrt{7/3}$. 
There exist
 $\delta\in (0,+\infty)$, $K>0$,
$\omega>0$
and a map  $g: M \to M^{\bot}$, where  $M^{\bot}\subset L^2(0,L)$ is  the orthogonal of $M$
for the $L^2$-scalar product,  satisfying
\begin{align}
&\label{gC3}
g\in C^3(
M; M^{\bot}),
\\
&g(0)=0,~ g'(0)=0,\label{g}
\end{align}
such that, with
\begin{align}
G:=\left\{ m+g\left( m\right);\; m\in M\right\}\subset L^2(0,L), \label{centerG}
\end{align}
the following three properties hold for every solution $y$ to \eqref{original}
with $\left\Vert y_0\right\Vert
_{L^2(0,L)}<\delta$,
\begin{enumerate}
 \item  (Local exponential attractivity of $G$.)
\begin{gather}
\label{exponentialattraction}
d( y(t,\cdot), G)\leq Ke^{-\omega t }d(y_0,G), ~\forall t>0,
\end{gather}
where $d(\chi,G)$ denotes the distance between $\chi\in L^2(0,L)$ and $G$:
\begin{align}
d(\chi,G):=\inf\{\|\chi-\psi\|_{L^2(0,L)};\;\psi\in G\}.
\end{align}

\item (Local invariance of $G$.)
\begin{equation}\label{eq-local-invariance}
\text{If $y_0\in G$, then
 $y(t,\cdot)\in G$, $\forall t\geq 0$.}
\end{equation}
\item If $y_0$ is in $G$, then
there exists 
$C>0$ such that\begin{align}
\Vert y(t,\cdot)\Vert _{L^2(0,L)}\leq \frac{C\Vert y_0\Vert _{L^2(0,L)}}{\sqrt{1+t\Vert y_0\Vert _{L^2(0,L)}^2}}, \quad \forall t \geq 0.
\label{estimatey(t)infinity}
\end{align}
\end{enumerate}
In particular, $0\in L^2(0,L)$ is locally asymptotically
stable in the sense
of the $L^2(0,L)$-norm for \eqref{original}.
\end{thm}
\begin{remark}\label{rem_nonunique}\rm
 It can be derived from \cite[Theorem 1 and Comments]{DN2014} that, for
every $L>0$,
  there are non-zero
stationary
solutions with the period of $L$ to the following ordinary differential equation
(ODE):
\begin{equation}
\left\{
\begin{array}
[l]{l}
f^{\prime}+ff^{\prime}+f^{\prime\prime
\prime}=0 \text{ in }[0,L],\\
f(0)=f(L)=0,\\
f'(L)=0.
\end{array}
\right.
\end{equation}
That is,   besides the origin, there also exist other steady states of the
nonlinear KdV equation (\ref{original}). Therefore,  $0\in
 L^{2}(0,L)$ is not globally asymptotically stable for  (\ref{original}): Property \hyperref[P1]{($\mathcal{P}_2$)}
 does not hold for arbitrary $\varepsilon_0>0$.
\end{remark}

 Our proof of  Theorem \ref{main result} relies on the center manifold approach. This center manifold is $G$ in Theorem \ref{main result}. Center manifold theory
plays an important role in studying
dynamic properties  of nonlinear systems near ``critical situations''. The  center manifold theorem was first proved for finite dimensional systems by Pliss \cite{P1964} and Kelley \cite{K1967},
and
the readers could refer to
 \cite{K2002, P2001} for more details of this theory.
Analogous results are also established for infinite dimensional systems, such as partial differential equations  (PDEs) \cite{C1981, BJ1989} and
functional differential equations
\cite{HI2010}.
The center manifold method usually leads
to a  dimension reduction of the original problems.
Then, in order to derive
stability properties (asymptotic stable, or, unstable) of the full nonlinear equations, one only needs to analyze
the reduced equation (restricted on the center manifold). When dealing with the infinite dimensional problems, this method can be extremely efficient  if the center manifold is finite dimensional. Following the results on existence, smoothness and attractivity of a center manifold for evolution equations   in  \cite{MW2004}, Chu, Coron and
Shang  studied in \cite{CCS2015}
the local asymptotic stability property of \eqref{original} with the critical
length
$L=2k\pi$ for any positive integer $k$
such that \eqref{condition-sur-pour-espapace-dim1} holds. They proved the existence of a one-dimensional local center manifold. By
analyzing the resulting one-dimensional reduced equation, they obtained the local asymptotic stability
of 0 for \eqref{original}. For $L=2\pi \sqrt{7/3}$, we get, following \cite{CCS2015}, the existence of a two-dimensional local center manifold. It is predictable that the two-dimensional local center manifold introduces more complexity than the one-dimensional local center manifold case.

The organization of this paper is as follows. In Section~\ref{preliminaries},
some basic properties of the linearized KdV equation \eqref{linearized0} and the
KdV equation \eqref{original} are given.
Then, in Section~\ref{sec-2-dim-center-manifold}, we recall a theorem on the existence of a local center manifold for the KdV equation (\ref{original}) and analyze  the dynamics
on the local center manifold.
Theorem \ref{main result} follows from this analysis. In Section~\ref{sec-conclusion}, we present the conclusion and some possible future works. Finally, we end this  article with an appendix that contains
computations which are important for the study of the dynamics on the center manifold.

\section{Preliminaries}\label{preliminaries}

\subsection{Some properties for the linearized equation of
(\ref{original}) around the origin}\label{proplinear}

The origin $y=0$ is an equilibrium of the  initial-boundary-value nonlinear KdV
problem  (\ref{original}). In this subsection, we derive some
properties for the linearized KdV
equation \eqref{linearized0}
around the origin of  (\ref{original}) posed on the finite interval $[0,L]$,
where $L=2\pi \sqrt{7/3}\in\mathcal{N}',$ for which there exists
a unique pair   $\{j=2,l=1\}$ satisfying \eqref{N'1}.

Let $\mathcal{A}: D\left(\mathcal{A}\right)\subset L^2(0,L) \rightarrow L^2(0,L)$
be the linear operator defined by
\begin{align}
\mathcal{A}\varphi:=-\varphi'-\varphi''',
\end{align}
with
\begin{align}
D(\mathcal{A}):=\left\{  \varphi \in H^{3}\left(  0,L\right);\; \varphi \left(  0\right)
=\varphi \left(  L\right)  =\varphi'\left(
 L\right)  =0\right\}  \subset L^2(0,L),
\end{align}
then the linearized equation \eqref{linearized0} can be written  as an evolution
equation in $L^2(0,L)$:
\begin{align}
&\frac{dy(t, \cdot)}{dt}=\mathcal{A}y(t, \cdot).
\end{align}
The following lemma can be immediately obtained.
 \begin{lem}\label{compact}
$\mathcal{A}^{-1}$ exists and is compact on $L^2(0,L)$. Hence,
$\sigma(\mathcal{A})$,
the spectrum of $\mathcal{A}$, consists of isolated eigenvalues only: $\sigma(\mathcal{A})=\sigma_p(\mathcal{A})$,
where $\sigma_p(\mathcal{A})$ denotes the set of eigenvalues of $\mathcal{A}$.
\end{lem}
\begin{proof}
By calculation, we get
\begin{equation}
\mathcal{A}^{-1}\varphi=\psi,~ \forall\varphi\in L^2(0,L),\\
\end{equation}
with
\begin{equation}
\psi : =-\frac{1-\cos(x-L)}{1-\cos L}\int_0^L (1-\cos y)\varphi(y)dy+\int_x^L
(1-\cos(x-y))\varphi(y)dy.
\end{equation}
Hence we get the existence of ${\mathcal{A}}^{-1}$ and that, by the Sobolev embedding
theorem, this operator is compact on $ L^2(0,L).$ Therefore, $\sigma(\mathcal{A})$, the spectrum of $\mathcal{A}$, consists of isolated eigenvalues only.
\end{proof}
The following proposition is proved.
\begin{prop}{\rm{(\cite[Proposition 3.1]{R1997})}}.
\label{LE1}
$\mathcal{A}$ generates a $C_{0}$-semigroup of contractions
$\left\{  S\left(
 t\right)  \right\}  _{t\geq0}$ on
$L^2(0,L)$, that  is, for any given
initial
data $y_{0}\in L^2(0,L)$, $S(t)y_{0}$ is the mild
solution of the linearized equation \eqref{linearized0}, and
\begin{align}
\left\Vert S(t)y_{0}\right\Vert _{L^2(0,L)}\leq\left\Vert y_{0}\right\Vert
_{L^2(0,L)  },\quad \forall t\geq0.
\end{align}
Moreover, for every $\lambda\in\sigma\left(\mathcal{A}\right)
 $, $\operatorname{Re}
\left(  \lambda\right)  \leq 0$.
\end{prop}
If
 $\operatorname{Re}
\left(  \lambda\right)  < 0$, $\forall \lambda\in\sigma\left(\mathcal{A}\right)$,
then it follows directly from the  ABLP (Arendt-Batty-Lyubich-Phong) Theorem \cite{N1993} that the semigroup $S(t)$ is asymptotically stable
on $L^2(0,L)$.
Since we only have $\operatorname{Re}
\left(  \lambda\right)  \leq 0$, $\forall \lambda\in\sigma\left(\mathcal{A}\right)$, the main concern needs to be put on  the
 eigenvalues on the imaginary axis and their corresponding eigenfunctions.
Following the proofs for \cite[Lemma 2.6]{CCS2015} and  \cite[Lemma 3.5]{R1997}, the following lemma is proved.
\begin{lem}
\label{Lm3} There exists a unique pair of conjugate eigenvalues of
$\mathcal{A}$
on the imaginary axis, that is,
\begin{align}
\sigma_{p}\left(  \mathcal{A}\right)  \cap i\mathbb{R}=\left\{\lambda=\pm iq;\;q= \frac{20}{21\sqrt{21}}\right\}.\label{q_value}
\end{align}Moreover, the corresponding eigenfunctions of $\mathcal{A}$ with respect to
$\lambda=\pm iq$ are
\begin{align}
\varphi:=C\left(\varphi_1\mp i\varphi_2\right),\label{phi1phi2}
\end{align}
respectively, where $C$ is an arbitrary constant, and
$\varphi_1,~\varphi_2$ are two nonzero real-valued functions:
\begin{align}
\varphi_1(x)= \Theta\left(\cos\left(\frac{5}{\sqrt{21}}x\right)
-3\cos\left(\frac{1}{\sqrt{21}}x\right)
+2\cos\left(\frac{4}{\sqrt{21}}x\right)\right),\label{phi1sol}\\
\varphi_2(x)=\Theta\left(-\sin\left(\frac{5}{\sqrt{21}}x\right)
-3\sin\left(\frac{1}{\sqrt{21}}x\right)+2\sin\left(\frac{4}{\sqrt{21}}x\right)\right),\label{phi2sol}
\end{align}
with
\begin{align}
\Theta:=\frac{1}{\sqrt{14\pi}} \sqrt[4]{\frac{3}{7}}.  \label{C}
\end{align}

\end{lem}

\begin{rem}
The equations satisfied by  $\varphi_1$ and
$\varphi_2$
are
\begin{equation}
\left\{
\begin{array}
[c]{c}
\varphi_1'+\varphi_1'''=-q\varphi_2,\\
\varphi_1(0)=\varphi_1(L)=0,\\
\varphi_1'(0)=\varphi_1'(L)=0,\label{eqn_varphi1}
\end{array}
\right.
\end{equation}
and
\begin{equation}
\left\{
\begin{array}
[c]{c}
\varphi_2'+\varphi_2'''=q\varphi_1,\\
\varphi_2(0)=\varphi_2(L)=0,\\
\varphi_2'(0)=\varphi_2'(L)=0.\label{eqn_varphi2}
\end{array}
\right.
\end{equation}
\end{rem}

\begin{rem}
We have
\begin{align}
\int_0^L \varphi_1 (x)\varphi_2(x)dx =0,
\label{innerproduct_phi}
\end{align}
and, with the definition of $\Theta$ given in \eqref{C},
\begin{equation}
\label{phi-i-norm-1}
\Vert \varphi_1\Vert _{L^2(0,L)}=\Vert \varphi_2\Vert _{L^2(0,L)}=1.
\end{equation}
\end{rem}

From the results in Lemma \ref{compact}, Proposition \ref{LE1} and Lemma \ref{Lm3},
we obtain
the following corollary.
\begin{cor}
\label{Corollary on spectrum}$\lambda=\pm i\frac{20}{21\sqrt{21}}$ is the unique
eigenvalue pair of $\mathcal{A}$
on the imaginary axis, and all the other eigenvalues of $\mathcal{A}$ have
negative
real parts which
are uniformly bounded away from the imaginary axis, i.e., there exists $r>0$
such that  any of the nonzero eigenvalues of $\mathcal{A}$ has a real part
which is less
than $-r$.
\end{cor}

Let us define
\begin{align}
M:={\text{span}}\{\varphi_1,\varphi_2\}=\{m_1 \varphi_1+m_2 \varphi_2;\; \mathbf{m}=(m_1,m_2) \in \R^2
\}\subset L^2(0,L),\label{spacedec}
\end{align}
where $\varphi_1,~\varphi_2$ are defined in \eqref{phi1sol}, \eqref{phi2sol} and \eqref{C}.
Then the
following decomposition holds:
\begin{align}
L^2(0,L)=M\oplus M^{\perp}, \label{dec_H}
\end{align}
with
\begin{align}
M^{\perp}:=\left\{\varphi\in L^2(0,L);\; \int_0^L \varphi(x)\varphi_1(x)dx=0,
\int_0^L \varphi(x)\varphi_2(x)dx=0\right\}.
\label{defMorthogonal}
\end{align}
\subsection{Some properties of the KdV equation (\ref{original})}\label{existenceCM}\

By considering the equation (\ref{original})  as a special
case (with $f=0$ and $u=0$) of the equation (4.6)--(4.8) in \cite{2007-Coron-book}, we give the following definition for a solution  to the equation (\ref{original}), which follows from \cite[Definition 4.1]{2007-Coron-book}.
\begin{defn}\label{def_kdv_sol}
Let $T > 0,~y_0\in L^2(0,L)$. A solution to the Cauchy problem
 (\ref{original}) on $[0,T]$ is a function 
 \begin{align}
 y \in \mathcal{B}:= C^0([0, T];L^2(0,L))\cap L^2(0,T;H^1(0,L))
 \end{align} 
 such that, for every
 $\tau \in [0,T]$ and for every $\phi \in C^3([0,\tau]\times [0,L])$ such that
\begin{align}
\phi(t, 0) =\phi(t,L)=\phi_x(t, 0) = 0,~\forall t\in[0,\tau],
\label{condition-bord-phi}
\end{align}
one has
\begin{align}\label{sol-def}
-\int_0^\tau\int_0^L (&\phi_t+\phi_x+\phi_{xxx})ydxdt+\int_0^\tau\int_0^L \phi yy_xdxdt\nonumber\\
&+\int_0^L y(\tau,x)\phi(\tau,x)dx-\int_0^L y_0( x)\phi(0,x)dx=0.
\end{align}
A solution to the Cauchy problem
 (\ref{original}) on $[0,+\infty)$ is a function 
 \begin{align}
 y \in C^0([0, +\infty);L^2(0,L))\cap L^2_{\text{loc}}([0,+\infty);H^1(0,L))
 \end{align} such that, for every $T>0$, $y$ restricted to $[0,T]\times (0,L)$ is a solution to (\ref{original}) on $[0,T]$.
\end{defn}
Then by considering
  equation (\ref{original})  as a special
case (with $f=0$ and $u=0$) of the equation $(A.1)$ in  \cite{CC2004}, the following two
propositions about  the existence and uniqueness
of the solutions to \eqref{original} follow directly from \cite[Proposition 14 and Proposition 15]{CC2004}.
\begin{prop}
\label{prop-existence-solution}
 Let $T\in (0,+\infty)$. There exist $\varepsilon=\varepsilon(T)>0$
and $C=C(T)>0$ such that, for every $y_0\in L^2(0,L)$ with $\Vert y_0\Vert _{L^2(0,L)}<\varepsilon(T)$,
there exists at least one solution $y$ to the equation  (\ref{original}) on $[0,T]$ which satisfies
\begin{align}
\|y\|_\mathcal{B}&:=\max_{t\in [0,T]} \|y(t,\cdot)\|_{L^2(0,L)}+\left(\int_0^T
\|y(t,\cdot)\|_{H^1(0,L)}^2dt\right)^{1/2}\nonumber\\
&\leq C(T)\|y_0\|_{L^2(0,L)}.
\end{align}
\end{prop}

\begin{prop}
\label{prop-difference}
 Let $T \in (0,+\infty)$. There exists $C>0$ such that,  for every solutions $y_1$ and $y_2,$ corresponding to every initial
conditions $(y_{10}, y_{20})\in
(L^2(0,L))^2$ respectively, to the equation   (\ref{original})
 on $[0,T]$, one has the following inequalities:
\begin{align}
&\int_0^T \int_0^L (y_{1x}(t,x)-y_{2x}(t,x))^2 dxdt\leq \int_0^L (y_{10}(x)-y_{20}(x))^2 dx\nonumber\\
&~~~~~~~~~~~~~~~~~~~~~\times
\exp\left(C\left(1+\|y_1\|^2_{L^2(0,T;H^1(0,L))}+\|y_2\|^2_{L^2(0,T;H^1(0,L))}\right)\right),\\
&\int_0^L (y_1(t,x)-y_2(t,x))^2 dx\leq \int_0^L (y_{10}(x)-y_{20}(x))^2
dx\nonumber\\
&~~~~~~~~~~~~~~~~~~~~\times \exp\left(C\left(1+\|y_1\|^2_{L^2(0,T;H^1(0,L))}+\|y_2\|^2_{L^2(0,T;H^1(0,L))}\right)\right),
\end{align}
for all $t\in [0,T]$.
\end{prop}
Let us also mention that for every solution $y$ to \eqref{original} on $[0,T]$ or on $[0,+\infty),$
\begin{equation}\label{energy-decreasing}
t\mapsto \|y(t,\cdot)\|_{L^2(0,L)}^2 \text{ is a non-increasing function.}
\end{equation}
This can be easily seen by multiplying the first equation of \eqref{original} with $y$, integrating on $[0,L]$ and performing integration by parts. One then gets, if $y$ is smooth enough,
\begin{equation}\label{derivative-of-energy}
\frac{d}{dt}\int_0^Ly(t,x)^2dx=-y_x(t,0)^2,
\end{equation}
which gives \eqref{energy-decreasing}. The general case follows from a smoothing argument. As a consequence of Proposition~\ref{prop-existence-solution},  Proposition~\ref{prop-difference} and  \eqref{energy-decreasing}, one sees that \eqref{original} has one and only one solution defined on $[0,+\infty)$ if $\|y_0\|_{L^2(0,L)}<\varepsilon(1)$.

\section{Existence of a center manifold and dynamics on this manifold}
\label{sec-2-dim-center-manifold}\

Let us start this section by recalling why, as it is classical, the property ``$0\in L^2(0,L)$ is locally asymptotically
stable in the sense
of the $L^2(0,L)$-norm for \eqref{original}'' stated at the end of Theorem~\ref{main result} is a consequence of the other statements
in this theorem. For convenience, let us recall the argument. Let $y_0\in L^2(0,L)$ be such that $\left\Vert y_0\right\Vert
_{L^2(0,L)}<\delta$ and let $y$ be the solution to \eqref{original}. It suffices to check that
\begin{equation}\label{ytendvers0}
  y(t,\cdot)\rightarrow 0 \text{ in } L^2(0,L) \text{ as } t\rightarrow +\infty.
\end{equation}
By \eqref{exponentialattraction}, \eqref{energy-decreasing} and the fact that $M$ is of finite dimension, there exists an increasing sequence of positive real numbers $(t_n)_{n\in \mathbb{N}}$ and
$z_0\in L^2(0,L)$ such that
\begin{gather}
\label{tntendinfty}
t_n\rightarrow +\infty \text{ as } n\rightarrow +\infty,
\\
\label{cvz0}
y(t_n,\cdot)\rightarrow z_0 \text{ in } L^2(0,L) \text{ as } n\rightarrow +\infty,
\\
\label{z0petitetdansG}
z_0\in G \text{ and } \left\Vert z_0\right\Vert
_{L^2(0,L)}<\delta.
\end{gather}
Let $z:[0,+\infty)\times (0,L)\to \R$ be the solution to \eqref{original} satisfying the initial condition $z(0,\cdot)=z_0$.
It follows from \eqref{estimatey(t)infinity} and \eqref{z0petitetdansG} that
\begin{gather}\label{cvdezvers0}
z(t,\cdot)\rightarrow 0 \text{ in } L^2(0,L) \text{ as } t\rightarrow +\infty.
\end{gather}
Let $\eta>0$. By \eqref{cvdezvers0}, there exists $\tau>0$ such that
\begin{equation}\label{zeta/2}
\left\Vert z(\tau,\cdot)\right\Vert_{L^2(0,L)}\leq \frac{\eta}{2}.
\end{equation}
By Proposition~\ref{prop-difference} and \eqref{cvz0},
\begin{equation}\label{cvny}
y(t_n+\tau,\cdot)\rightarrow z(\tau,\cdot) \text{ in } L^2(0,L) \text{ as } n\rightarrow +\infty.
\end{equation}
By \eqref{zeta/2} and \eqref{cvny}, there exists $n_0\in \mathbb{N}$ such that
\begin{equation}\label{cvnyeta}
\left\Vert y(t_{n_0}+\tau,\cdot)\right\Vert_{L^2(0,L)}< \eta,
\end{equation}
which, together with \eqref{energy-decreasing}, implies that
\begin{equation}\label{y(t)petit}
\left\Vert y(t,\cdot)\right\Vert_{L^2(0,L)}< \eta,\quad \forall t\geq t_{n_0}+\tau,
\end{equation}
which  concludes the proof of \eqref{ytendvers0}.

The remaining parts of this section are organized as follows. We first recall in Section~\ref{existence_CM} a theorem (Theorem~\ref{center-manifold_0}) on the existence of a local center manifold for \eqref{original}. Then in Section~\ref{subsecdynamics-center-manifold} we analyze the dynamics of \eqref{original} on this center manifold
and deduce Theorem~\ref{main result} from this analysis.

\subsection{Existence of a local center manifold}\label{existence_CM}\

In \cite[Theorem 3.1]{CCS2015}, following \cite{MW2004}, the existence of a center manifold for \eqref{original} was proved for the first critical length, i.e., $L=2\pi$. The same proof applies for our $L$ (i.e., the $L$ defined by \eqref{valueL}) and allows us to get the following theorem.

\begin{thm}\label{center-manifold_0}
There exist $\delta\in (0,\varepsilon(1))$, $K>0$, $\omega>0$
and a map  $g: M \to M^{\bot}$  satisfying \eqref{gC3} and \eqref{g}
such that, with $G$ defined by \eqref{centerG},
the following two   properties  hold for every solution $y(t,x)$ to \eqref{original}
with
$\left\Vert y_0\right\Vert
_{L^2(0,L)}<\delta$,
\begin{enumerate}
 \item  (Local exponential attractivity of $G$.)
\begin{gather}
\label{exponentialattraction-new}
d( y(t,\cdot), G)\leq Ke^{-\omega t }d(y_0,G), ~\forall t>0,
\end{gather}
where $d(\chi,G)$ denotes the distance between $\chi\in L^2(0,L)$ and $G$:
\begin{align}
d(\chi,G):=\inf\{\|\chi-\psi\|_{L^2(0,L)};\;\psi\in G\}.
\end{align}

\item (Local invariance of $G$.)
\begin{equation}\label{eq-local-invariance-new}
\text{If $y_0\in G$, then
 $y(t,\cdot)\in G$, $\forall t\geq 0$.}
\end{equation}
\end{enumerate}
\end{thm}

\subsection{Dynamics on the local center manifold}
\label{subsecdynamics-center-manifold}\

In this section we study the dynamics of \eqref{original} on $G_{\delta}$ with
\begin{gather}
\label{defGdelta}
G_{\delta}:=\{\zeta(x)\in G;\; \left\Vert\zeta\right\Vert_{L^2(0,L)}< \delta\}.
\end{gather}
Let
\begin{gather}
\label{defOmega}
\Omega:=\{(m_1,m_2) \in \R^2;\, m_1\varphi_1+m_2\varphi_2 + g(m_1\varphi_1+m_2\varphi_2)\in G_{\delta}\} ,
\end{gather}
then $\Omega$ is a bounded open subset of $\R^2$ which contains $(0,0)\in \R^2$. Let $\mathbf{m}^0=(m_1^0,m_2^0)\in \Omega$, and let $y$ be the solution of \eqref{original} on $[0,+\infty)$ for the initial data $y_0:=m_1^0\varphi_1+m_2^0\varphi_2 + g(m_1^0\varphi_1+m_2^0\varphi_2)$.
It follows from \eqref{energy-decreasing} and Theorem~\ref{center-manifold_0} that $y(t,\cdot)\in G_{\delta}$ for every $t\in[0,+\infty)$. Hence we can define, for $t\in [0,+\infty)$, $\mathbf{m}(t)=(m_1(t),m_2(t))\in \Omega$ by requiring that
\begin{equation}\label{defm(t)}
y(t,\cdot)=m_1(t)\varphi_1+m_2(t)\varphi_2 + g(m_1(t)\varphi_1+m_2(t)\varphi_2).
\end{equation}
Since $y\in C^0([0,+\infty); L^2(0,L))$, then $\mathbf{m}\in C^0([0,+\infty);\R^2)$.
Let $T>0$. Let $u\in C^\infty_0(0,T)$. We apply \eqref{sol-def} with $\tau=T$ and $\phi(t,x):=u(t)\varphi_1(x)$ (note that, by \eqref{eqn_varphi1}, \eqref{condition-bord-phi} holds). We get
\begin{align}
-\int_{0}^{T} \int_0^L (\dot u(t) \varphi_1(x) + u(t)\varphi'_1(x) + u(t)\varphi'''_1(x))y(t,x) dxdt
\nonumber\\
+\int_{0}^{T} \int_0^L u(t)\varphi_1(x)(yy_x)(t,x)  dxdt=0.
\label{dotm-1-step1}
\end{align}
 From \eqref{eqn_varphi1}, \eqref{defMorthogonal}, \eqref{defm(t)} and \eqref{dotm-1-step1}, we have
\begin{equation}
-\int_{0}^{T}  (m_1(t)\dot u(t)  -q m_2(t)u(t)) dt -\frac{1}{2}\int_{0}^{T} \int_0^L y^2(t,x)\varphi'_1(x) u(t) dxdt=0.
\label{dotm-1-step2}
\end{equation}
Hence, in the sense of distributions on $(0,T)$,
\begin{equation}
 \dot m_1= -q m_2 + \frac{1}{2} \int_0^L \left(m_1\varphi_1+m_2\varphi_2+g\left(m_1\varphi_1+m_2\varphi_2\right)\right)^2\varphi'_1dx.
\label{dotm-1-step3}
\end{equation}
Similarly, in the sense of distributions on $(0,T)$,
\begin{equation}
 \dot m_2= q m_1 +\frac{1}{2}\int_0^L \left(m_1\varphi_1+m_2\varphi_2+g\left(m_1\varphi_1+m_2\varphi_2\right)\right)^2\varphi'_2dx.
\label{dotm-1-step4}
\end{equation}
Hence, if we define $F:\Omega\to \R^2$, $\mathbf{m}=(m_1, m_2) \mapsto
F(\mathbf{m})$, by
\begin{gather}\label{defF}
F(\mathbf{m}):=
\begin{pmatrix}
\displaystyle -q m_2 +\frac{1}{2} \int_0^L
(m_1\varphi_1+m_2\varphi_2 + g(m_1\varphi_1+m_2\varphi_2))^2\varphi'_1dx
\\
\displaystyle q m_1 +\frac{1}{2} \int_0^L
(m_1\varphi_1+m_2\varphi_2 + g(m_1\varphi_1+m_2\varphi_2))^2\varphi'_2dx
\end{pmatrix}
,
\end{gather}
then
\begin{equation}\label{dotmathbfm}
\dot{\mathbf{m}}=F(\mathbf{m}).
\end{equation}
Note that, by \eqref{gC3} and \eqref{defF},
\begin{equation}\label{FC3}
  F\in C^3(\Omega;\R^2),
\end{equation}
which, together with \eqref{dotmathbfm}, implies that
\begin{equation}\label{mregular}
  \mathbf{m}\in C^4([0,+\infty);\R^2).
\end{equation}

We now estimate $g$ close to $0\in M$. Let $\psi \in C^3([0,L])$ be such that
\begin{equation}
\psi (0) = \psi (L) = \psi'(0) = 0.
\label{condition-bord-psi}
\end{equation}
  Using Definition~\ref{def_kdv_sol} with $\phi(t,x):= \psi(x)$, \eqref{condition-bord-psi} and  integration by parts, we get
\begin{align}
-\frac{1}{\tau}\int_{0}^{\tau} \int_0^L (\psi'+ \psi''')y dxdt -\frac{1}{2\tau }
\int_{0}^{\tau} \int_0^L \psi'y^2  dxdt  \nonumber\\
 + \int_0^L \frac{1}{\tau}\left(y(\tau,x)-y_0(x)\right) \psi(x) dx  = 0.
 \label{eq-psi-int}
\end{align}
Letting $\tau\rightarrow 0^+$ in \eqref{eq-psi-int}, and using \eqref{defF}, \eqref{dotmathbfm} and \eqref{mregular}, we get
\begin{align}
-\int_0^L (\psi'+ \psi''')y_0 dx -\frac{1}{2}
\int_0^L \psi' y^2_0  dx + \int_0^L \Big(\dot m_1(0)\varphi_1(x) + \dot m_2(0)\varphi_2(x)  \nonumber\\
+\frac{\partial g}{\partial m_1}( \mathbf{m}^0) \dot m_1(0) + \frac{\partial g}{\partial m_2}( \mathbf{m}^0)\dot m_2(0)\Big)\psi dx  = 0.
 \label{eq-psi-int-tau=0}
\end{align}
We expand $g$ in a neighborhood of $0\in M$. Using \eqref{gC3} and \eqref{g}, there exist 
\begin{align}
a\in M^\perp, b\in M^\perp, c\in M^\perp\label{abcperp}
\end{align} 
such that
\begin{align}
&g(\alpha\varphi_1+\beta\varphi_2) = 
\alpha^2 a+\alpha \beta b+ \beta^2c+o(\alpha^2+\beta^2 ) \text { in } L^2(0,L) \text{ as } \alpha^2+\beta^2\rightarrow 0, \label{estimation_ystar_order2}\\
&\frac{\partial g}{\partial m_1}(\alpha\varphi_1+\beta\varphi_2)= 
2\alpha a + \beta b +o(|\alpha|+|\beta|) \text { in } L^2(0,L) \text{ as } |\alpha|+|\beta| \rightarrow 0, \label{estimation_ystar-1_order2}\\
&\frac{\partial g}{\partial m_2}(\alpha\varphi_1+\beta\varphi_2)= 
\alpha b + 2\beta c +o(|\alpha|+|\beta|) \text { in } L^2(0,L) \text{ as } |\alpha|+|\beta| \rightarrow 0. \label{estimation_ystar-2_order2}
\end{align}
As usual, by \eqref{estimation_ystar_order2}, we mean that, for every $\varsigma_1>0$, there exists $\varsigma_2>0$ such that
\begin{align}\label{defo}
&\left(\alpha^2+\beta^2\leq \varsigma_1\right)
\nonumber\\
&\Rightarrow\left(\Vert g(\alpha\varphi_1+\beta\varphi_2) - \left(\alpha^2 a+\alpha \beta b+ \beta^2c\right)\Vert_{L^2(0,L)}\leq \varsigma_2(\alpha^2+\beta^2 )\right).
\end{align}
Similar definitions are used in \eqref{estimation_ystar-1_order2}, \eqref{estimation_ystar-2_order2} and later on.
We now expand the left hand side of \eqref{eq-psi-int-tau=0} in terms of $m_1^0$, $m_2^0$, $(m_1^0)^2$, $m_1^0m_2^0$ and
$(m_2^0)^2$ as $|m_1^0|+|m_2^0|\rightarrow 0$.

 For the functions $\varphi_1$ and $\varphi_2$ defined by \eqref{phi1sol}, \eqref{phi2sol} and  \eqref{C}, the following equalities can be derived from  \eqref{eqn_varphi1},
\eqref{eqn_varphi2}
and using integrations by parts:
\begin{align}
&\int_{0}^{L}\varphi_1(x)\varphi_2'(x)dx=\frac{10}{7\sqrt{21}},
~~~~~~\int_{0}^{L}\varphi_2(x)\varphi_1'(x)dx    =-\frac{10}{7\sqrt{21}},
\label{12prime}\\
&\int_{0}^{L}\varphi_1^{2}(x)\varphi_1'(x)dx    =0,~~
~~~~~~~~~~\int_{0}^{L}\varphi_2^{2}(x)\varphi_2'(x)dx    =0,\label{phi^2phix}\\
&\int_{0}^{L}\varphi_{1}^2(x)\varphi_{2}'(x)dx     =-2c_1,~~~~~~\int_{0}^{L}\varphi_{2}^2(x)\varphi_{1}'(x)dx
=2\sqrt{3}c_1,\\
&\int_{0}^{L}\varphi_{1}(x)\varphi_{2}(x)\varphi_{1}'(x)dx     =c_1,~~~\int_{0}^{L}\varphi_{1}(x)\varphi_{2}(x)\varphi_{2}'(x)dx
  =-\sqrt{3}c_1,
\label{phiy*y*x1}
\end{align}
where the constant $c_1$ is defined by
\begin{align}
c_1:=\frac{177147}{392392\pi}\sqrt{\frac{1}{2\pi}}\sqrt[4]{\frac{3}{7}}.\label{c1_value}
\end{align}

Looking successively at the terms in $(m_1^0)^2$, $m_1^0m_2^0$ and
$(m_2^0)^2$ in \eqref{eq-psi-int-tau=0} as $|m_1^0|+|m_2^0|\rightarrow 0$, we get, using  \eqref{defF}, \eqref{dotmathbfm}, \eqref{estimation_ystar_order2}, \eqref{estimation_ystar-1_order2},  \eqref{estimation_ystar-2_order2} as well as \eqref{12prime}--\eqref{phiy*y*x1},
\begin{gather}
-\int_0^L (\psi_x + \psi_{xxx})a dx -\frac{1}{2}
\int_0^L \psi_x \varphi_1^2  dx
  + \int_0^L \left(-c_1 \varphi_2 + q b\right)\psi dx  = 0,
\label{m-1^2}
\end{gather}
\begin{align}
-\int_0^L (\psi_x + \psi_{xxx})b dx &-
\int_0^L \psi_x \varphi_1\varphi_2 dx
\nonumber\\
&+ \int_0^L \left(c_1 \varphi_1 -\sqrt{3}c_1\varphi_2-2 qa +2 q c\right)\psi dx=0,
\label{m-1m-2}
\end{align}
\begin{gather}
\label{m-2^2}
-\int_0^L (\psi_x + \psi_{xxx})c dx-\frac{1}{2}
\int_0^L \psi_x \varphi_2^2  dx
+ \int_0^L \left(\sqrt{3} c_1 \varphi_1 - q b\right)\psi dx=0.
\end{gather}
Since \eqref{m-1^2}, \eqref{m-1m-2} and \eqref{m-2^2} must hold
for every $\psi\in C^3([0,L])$ satisfying \eqref{condition-bord-psi}, one gets that $a$, $b$ and $c$ are of class
$C^\infty$ on $[0,L]$ and satisfy

\begin{equation}
\left\{
\begin{array}
[c]{l}
a'+a'''+\varphi_1\varphi_{1}'-c_1\varphi_2+qb=0,\\
a(0)=a(L)=0,~a'(L)=0,\label{a21}
\end{array}
\right.  
\end{equation}
\begin{equation}
\left\{
\begin{array}
[c]{l}
b'+b'''+\varphi_{1}\varphi_{2}'+\varphi_{1}'\varphi_{2}+c_1\varphi_1-\sqrt{3}c_1\varphi_2 -2qa+2qc=0,\\
b(0)=b(L)=0,~b'(L)=0,
\end{array}
\right.  \label{a22}
\end{equation}
\begin{equation}
\left\{
\begin{array}
[c]{l}
c'+c'''+\varphi_2\varphi_{2}'+\sqrt{3}c_1\varphi_1-qb=0,\\
c(0)=c(L)=0,~c'(L)=0.
\end{array}
\right.  \label{a23}
\end{equation}

 We derive in the Appendix  the unique functions $a:[0,L]\to\R$, $b:[0,L]\to\R$ and $c:[0,L]\to\R$ which are  solutions
to  \eqref{a21}, \eqref{a22} and \eqref{a23}.
 From \eqref{defF} and \eqref{estimation_ystar_order2}, we get that, as $\mathbf{m}\rightarrow \mathbf{0}\in \R^2$,
\begin{align}
F(\mathbf{m})=
&\begin{pmatrix}
-qm_{2}+\sqrt{3}
c_{1}m_{2}^{2}+c_{1}m_{1}m_{2}+A_{1}m_{1}^{3}+B_{1}m_{1}^{2}m_{2}+C_{1}m_{1}m_{2}^{2}+D_{1}m_{2}^{3}\\
qm_{1}-c_{1}m_{1}^{2}-\sqrt{3}
c_{1}m_{1}m_{2}+A_{2}m_{1}^{3}+B_{2}m_{1}^{2}m_{2}+C_{2}m_{1}m_{2}^{2}+D_{2}m_{2}^{3}
\end{pmatrix}\nonumber\\&
+o(|\mathbf{m}|^3),
\label{ODE}
\end{align}
with
\begin{gather}A_1:=\int_0^La\varphi_1\varphi_1'dx,
\label{defA1}
\\
B_1:=\int_0^Lb\varphi_1\varphi_1'dx+\int_0^La\varphi_2\varphi_1'dx,
\label{defB1}
\\
C_1:=\int_0^Lc\varphi_1\varphi_1'dx+\int_0^Lb\varphi_2\varphi_1'dx,
\label{defC1}
\\
D_1:=\int_0^Lc\varphi_2\varphi_1'dx,
\label{defD1}
\\
A_2:=\int_0^La\varphi_1\varphi_2'dx,
\label{defA2}
\\
B_2:=\int_0^Lb\varphi_1\varphi_2'dx+\int_0^La\varphi_2\varphi_2'dx,
\label{defB2}
\\
C_2:=\int_0^Lc\varphi_1\varphi_2'dx+\int_0^Lb\varphi_2\varphi_2'dx,
\label{defC2}
\\
D_2:=\int_0^Lc\varphi_2\varphi_2'dx.
\label{defD2}
\end{gather}

Let us now study the local asymptotic stability property of $\mathbf{0}\in \R^2$ for \eqref{dotmathbfm}. We propose two methods for that. The first one is a more direct one, which relies on normal forms for dynamical systems on $\R^2$. The second one, which
relies on a Lyapunov approach related to the physics of \eqref{original}, is less direct. However, there is a reasonable hope that this second method can be applied
to other critical lengths $ L\in \mathcal{N} \setminus 2\pi\mathbb{N}$ for which the dimension of $M$ is larger than $2$.

\textbf{Method 1: normal form.} Let
\begin{equation}
z:=m_{1}+im_{2}\in \mathbb{C}.
\label{defz}
\end{equation}
Then
\begin{equation}
m_{1}=\frac{z+\overline{z}}{2},m_{2}=\frac{z-\overline{z}}{2i},
\end{equation}
and it follows from \eqref{dotmathbfm} and (\ref{ODE}) that, as $|z|\rightarrow 0$,
\begin{equation}
\dot{z}=\left( iq\right) z+P_{2}(z,\overline{z})+P_{3}(z,
\overline{z})+o(|z|^3),  \label{complex form}
\end{equation}
where $P_{j}(z,\overline{z})$ are polynomials in $z,\overline{z}$ of degree $
j.$ To be more precise, we have
\begin{eqnarray}
P_{2}(z,\overline{z}) &:=&\left( \sqrt{3}c_{1}m_{2}^{2}+c_{1}m_{1}m_{2}
\right) +i\left( -c_{1}m_{1}^{2}-\sqrt{3}c_{1}m_{1}m_{2}\right)  \notag \\
&=&-\frac{c_{1}}{2}\left( \sqrt{3}+i\right) z^{2}+\frac{c_{1}}{2}\left(
\sqrt{3}-i\right) z\overline{z},  \label{F2}
\end{eqnarray}
and
\begin{align}
\displaystyle P_{3}(z,\overline{z}):=\left( A_{1}+iA_{2}\right) \left( \frac{z+\overline{z}
}{2}\right) ^{3}+\left( B_{1}+iB_{2}\right) \left( \frac{z+\overline{z}}{2}
\right) ^{2}\left( \frac{z-\overline{z}}{2i}\right) \nonumber\\
\displaystyle
\text{ \ \ \ \ \ \ \ \ \ }+\left( C_{1}+iC_{2}\right) \left( \frac{z+
\overline{z}}{2}\right) \left( \frac{z-\overline{z}}{2i}\right) ^{2}+\left(
D_{1}+iD_{2}\right) \left( \frac{z-\overline{z}}{2i}\right) ^{3}.
\label{F3}
\end{align}
We can rewrite (\ref{complex form}) as
\begin{equation}
\dot{z}=\left( iq\right) z+\sum\limits_{i+j=2}^{3}\frac{1}{i!j!}
g_{ij}z^{i}\overset{-}{z}^{j}+o(|z|^{3}),  \label{standard form}
\end{equation}
and it is known from \cite[page 45 and page 47]{1981-Hassard-et-al-book} that (\ref{standard form}
) has the following Poincar\'{e} normal form
\begin{equation}
\dot{\xi }=\left( iq\right) \xi +\rho \xi ^{2}\overline{\xi }
+o(|\xi |^{3}),  \label{normal form}
\end{equation}
where
\begin{equation}
\rho =\frac{i}{2q}\left(g_{20}g_{11}-2\left\vert g_{11}\right\vert ^{2}-\frac{1}{3
}\left\vert g_{02}\right\vert ^{2}\right)+\frac{g_{21}}{2}.  \label{pho}
\end{equation}
According to (\ref{F2}) and (\ref{F3}), through a simple computation, we
have
\begin{gather}
g_{20}=-c_{1}\left( \sqrt{3}+i\right) ,g_{11}=\frac{c_{1}}{2}\left( \sqrt{3}
-i\right) ,g_{02}=0,  \label{part2}
\\
g_{21}=\frac{1}{4}\left( 3A_{1}+i3A_{2}
-i B_{1}+ B_{2} + C_{1}+iC_{2} +
-i3D_{1}+3D_{2}\right) .
\label{part3}
\end{gather}
Using (\ref{part2}) and (\ref{part3}), the formula of $\rho $ provided by (\ref{pho}) gives
\begin{eqnarray}
\rho =\rho_1+i \rho_2,
\end{eqnarray}
with
\begin{gather}
\label{defa}
\rho_1:=\frac{1}{8}\left(3A_{1}+C_{1}
+B_{2}+3 D_{2}\right), \\
 \rho_2:=
-2\frac{c_{1}^{2}}{q}+\frac{1}{8}\left(-B_{1} -3 D_{1} +3 A_{2}+C_{2}\right) .
\label{defb}
\end{gather}
It follows that we can derive  the Poincar\'{e} normal form of the
reduced equation on the local center manifold (\ref{normal form}). Moreover, in Cartesian coordinates, (\ref{normal form})  is
\begin{eqnarray}
\dot{\xi}_1 &=&-q\xi_2+(a\xi_1-b\xi_2)\left( \xi_1^{2}+\xi_2^{2}\right) +o(|\xi_1|^3+|\xi_2|^3), \\
\dot{\xi}_2 &=&q\xi_1+(a\xi_2+b\xi_1)\left( \xi_1^{2}+\xi_2^{2}\right) +o(|\xi_1|^3+|\xi_2|^3),
\end{eqnarray}
where
\begin{equation}
\xi =\xi_1+i \xi_2.
\end{equation}
In polar coordinates, set
\begin{equation}
r=\sqrt{\xi_1^{2}+\xi_2^{2}},\theta =\arctan \frac{\xi_2}{\xi_1},
\end{equation}
we have, as $r\rightarrow 0$,
\begin{eqnarray} \label{polar}
\dot{r} =\rho_1 r^{3}+o(r^{3}),\;
\dot{\theta } =q+\rho_2 r^{2}+o(r^{2}).
\end{eqnarray}
Now it is clear to see from (\ref{polar}) that the origin $\mathbf{0}\in \R^2$ is
asymptotically stable for \eqref{dotmathbfm} if $\rho_1<0$ and is not stable if $\rho_1>0$. From
\eqref{phi1sol}, \eqref{phi2sol}, \eqref{C}, \eqref{defA1}--\eqref{defD2} and the Appendix, we can obtain all the coefficients $A_{i},B_{i},C_{i},D_{i}$ $(i=1,2) $. Then, using Matlab, it follows that
\begin{equation}
 \rho_1:=\frac{1}{8}\left(A_{1}+C_{1}
+B_{2} +3 D_{2}\right) =-0.014325<0.
\end{equation}
And straightforward computation leads to the existence of $C>0$ such that, at least if $r(0)\in [0,+\infty)$ is small enough, 
one has for the solution to \eqref{polar},
\begin{equation}\label{estimate-r(t)} 
r(t)\leq \frac{C r(0)}{\sqrt{1+tr(0)^2}}, \quad \forall t\in [0,+\infty),
\end{equation}
which concludes the proof of Theorem~\ref{main result}.

\textbf{Method 2: Lyapunov function.} Let us start with a formal motivation. Recall  that, by \eqref{globallstabb} and  with $E$ defined in \eqref{defE},
we have, along the trajectories of \eqref{original},
\begin{gather}
\dot E=-\frac{1}{2}K^2
\label{dotEK},
\end{gather}
with
\begin{gather}
K:= y_x(0).
\label{defK}
\end{gather}
It is therefore natural to consider the following candidate for a Lyapunov function
\begin{equation}\label{defV}
V:=E-\mu K \dot K,
\end{equation}
where $\mu >0$ is small enough. Indeed, one then gets 
\begin{equation}\label{dotVmu}
\dot V:=-\frac{1}{2}K^2-\mu \left(\dot K\right)^2 -\mu K \ddot K,
\end{equation}
and one may hope to absorb $-\mu K \ddot K$ with $-(1/2)K^2-\mu \left(\dot K\right)^2$ and get 
$\dot V<0$ on $G\setminus\{0\}$, at least in a neighborhood of $0$. 

We follow this strategy together with the approximation of $g$ previously found. For $\mathbf{m}=(m_1,m_2)\in \Omega$, let (see \eqref{estimation_ystar_order2})
\begin{gather}
\label{deftildey}
\tilde y =m_1\varphi_1+m_2\varphi_2+m_1^2 a+m_1m_2 b+ m_2^2 c \in C^\infty([0,L]),
\\
\label{tildeE}
\tilde E :=\frac{1}{2}\int_0^L  \tilde y ^2 dx.
\end{gather}
 Then, using \eqref{eqn_varphi1}, \eqref{eqn_varphi2}, \eqref{a21}, \eqref{a22} and \eqref{a23} (compare with \eqref{eq-psi-int-tau=0}), one gets that, along the trajectories of \eqref{dotmathbfm}, for $\mathbf{m}\in \Omega$ and $\psi \in C^3([0,L])$ satisfying
 \begin{equation}
\psi (0) = \psi (L)= 0,
\label{condition-bord-psi-sans-prime}
\end{equation}
one has
\begin{align}
&-\int_0^L (\psi' + \psi''')\tilde y dx\nonumber\\ 
&+ \psi'(0)\left(m_1^2 a'(0)+m_1m_2 b'(0)+ m_2^2 c'(0)\right)
\nonumber\\
&-\frac{1}{2}
\int_0^L \psi_x \tilde y ^2  dx
  + \int_0^L \left(\dot m_1\varphi_1 + \dot m_2\varphi_2+
 \frac{\partial \tilde g}{\partial m_1} \dot m_1
+ \frac{\partial \tilde g}{\partial m_2}\dot m_2\right)\psi dx
\nonumber\\ 
&=\int_0^L (\tilde y_{t}+ \tilde y_x+\tilde y_{xxx}+\tilde y\tilde y_{x})\psi
dx\nonumber\\
&=\int_0^L  \bigg[m_1^3(A_1\varphi_1+A_2\varphi_2-bc_1+\varphi_1a'+a\varphi_1')\nonumber\\
&~~~~~~~~~~~+m_1^2m_2(B_1\varphi_1+B_2\varphi_2+2ac_1-b\sqrt{3}c_1\nonumber\\
&~~~~~~~~~~~~~~~~~~~~~~-2cc_1+\varphi_1b'+\varphi_2a'+a\varphi_2'+b\varphi_1')\nonumber\\
&~~~~~~~~~~~+m_1m_2^2(C_1\varphi_1+C_2\varphi_2+2a\sqrt{3}
c_{1}+bc_1\nonumber\\
&~~~~~~~~~~~~~~~~~~~~~~-2c\sqrt{3}
c_{1}+\varphi_1c'+\varphi_2b'+b\varphi_2'+c\varphi_1')\nonumber\\
&~~~~~~~~~~~+m_2^3(D_1\varphi_1+D_2\varphi_2+b\sqrt{3}
c_{1}+\varphi_2c'+c\varphi_2')\nonumber\\
&~~~~~~~~~~~+o(|\mathbf{m}|^3)\bigg]\psi
dx  \text{ as } |\mathbf{m}|\rightarrow 0.
 \label{formule-integrale-tilde}
\end{align}
Then, using \eqref{formule-integrale-tilde} with $\psi:=\tilde y
$ (which, by \eqref{eqn_varphi1}, \eqref{eqn_varphi2}, \eqref{a21}, \eqref{a22}, \eqref{a23} and  \eqref{deftildey}, satisfies \eqref{condition-bord-psi-sans-prime}), along the trajectories of \eqref{dotmathbfm}, we have
from \eqref{innerproduct_phi}, \eqref{phi-i-norm-1}, 
\eqref{abcperp}  and \eqref{ODE}--\eqref{defD2} that the right hand side of \eqref{formule-integrale-tilde} is $o(|\mathbf{m}|^4)$, and
\begin{align}
\label{exptildeE}
\dot{\tilde E}=&  -\frac{1}{2}\tilde K^2 + o(|\mathbf{m}|^4) \text{ as } |\mathbf{m}| \rightarrow 0,
\end{align}
with $\tilde K:\Omega \rightarrow \R$ defined by
\begin{equation}\label{deftildeK}
\tilde K: =a'(0) m_1^2 +b'(0) m_1m_2 + c'(0) m_2^2.
\end{equation}
Let us emphasize that, even if  ``along the trajectories of \eqref{dotmathbfm}'' might be misleading, $\dot{\tilde E}$ is just a function of $\mathbf{m}\in \Omega$. It is the same for $\dot{\tilde V}$, $\dot{\tilde K}$, $\ddot{\tilde K}$ which appear below. Using \eqref{g} and \eqref{defF}, we have, along the trajectories of \eqref{dotmathbfm},
\begin{equation}\label{dottildeKK}
\dot{\tilde K}=qb'(0)m_1^2+2q(c'(0)-a'(0))m_1m_2-qb'(0)m_2^2+o(|\mathbf{m}|^2).
\end{equation}
Using \eqref{defF}, we get the existence of $C>0$ such that, along the trajectories of \eqref{dotmathbfm},
\begin{equation}\label{ddottildeK}
\left|\ddot {\tilde K}\right|\leqslant C |\mathbf{m}|^2,\quad \forall \mathbf{m} \in \Omega.
\end{equation}

We can now define our Lyapunov function $\tilde V$. Let $\mu\in (0,1/4]$. Let $\tilde V :\Omega \to \R$ be defined by
\begin{equation}
\label{deftildeV}
\tilde V:= \tilde E-\mu \tilde K \dot {\tilde K}.
\end{equation}
 From \eqref{deftildeV}, we have the existence of $\eta_0>0$ such that, for every $\mathbf{m}\in \R^2$  satisfying $|\mathbf{m}|<\eta_0$ and along the trajectories of \eqref{dotmathbfm},
\begin{eqnarray}
\label{eqdottildeV}
\dot{\tilde V}&=& -\frac{1}{2}\tilde K^2-\mu \left(\dot {\tilde K}\right)^2-\mu \tilde K \ddot {\tilde K}+ o(|\mathbf{m}|^4)
\nonumber\\
&\leq& -\frac{1}{4}\tilde K^2 -\mu \left(\dot {\tilde K}\right)^2+\mu^2\left(\ddot {\tilde K}\right)^2+ o(|\mathbf{m}|^4)\nonumber\\
&\leq& -\frac{1}{4}\tilde K^2 -\mu \left(\dot {\tilde K}\right)^2+ 2\mu^2C ^2|\mathbf{m}|^4 
\nonumber\\
&\leq& -\mu\left(\tilde K^2 + \left(\dot {\tilde K}\right)^2- 2\mu C ^2|\mathbf{m}|^4\right).
\end{eqnarray}
 Let us assume for the moment that, for every $\mathbf{m}=(m_1,m_2)\in \R^2$,
\begin{equation}\label{not-degenerate}
\left(
\left\{
\begin{array}{l}
a'(0) m_1^2 +b'(0) m_1m_2 + c'(0) m_2^2=0,
\\
qb'(0)m_1^2+2q(c'(0)-a'(0))m_1m_2-qb'(0)m_2^2=0,
\end{array}
\right.
\right)
\Rightarrow
\left(
\mathbf{m}=\mathbf{0}
\right).
\end{equation}
 Then, by homogeneity, there exists $\eta_1>0$ such that
\begin{align}\label{lowersomme}
\left(a'(0) m_1^2 +b'(0) m_1m_2 + c'(0) m_2^2\right)^2&+\left(qb'(0)m_1^2+2q(c'(0)-a'(0))m_1m_2-qb'(0)m_2^2\right)^2
\nonumber\\
&\geq
2 \eta_1 |\mathbf{m}|^4, \, \forall \mathbf{m}=(m_1,m_2)\in \R^2.
\end{align}
 From \eqref{deftildeK}, \eqref{dottildeKK} and \eqref{lowersomme},  we get the existence of $\eta_2>0$ satisfying
\begin{equation}\label{lowersomme-K}
\tilde K^2 +\left(\dot {\tilde K}\right)^2\geq \eta_1 |\mathbf{m}|^4, \quad \forall \mathbf{m} \in \R^2
\text{ such that } |\mathbf{m}|<\eta_2.
\end{equation}
 From \eqref{eqdottildeV} and \eqref{lowersomme-K}, we get the existence of $\eta_3>0$ such that, for every
 $\mu\in (0,\eta_3)$,
\begin{equation}\label{estimate-dottildeV}
\dot{\tilde V}\leq - \frac{\mu}{2}\eta_1 |\mathbf{m}|^4 , \quad \forall \mathbf{m} \in \R^2
\text{ such that } |\mathbf{m}|<\eta_3.
\end{equation}
Moreover, straightforward estimates show that there exists $\eta_4>0$ such that, for every
 $\mu\in (0,\eta_4)$,
\begin{equation}\label{estimate-V}
\eta_4 |\mathbf{m}|^2 \leq \tilde V \leq \frac{1}{\eta_4} |\mathbf{m}|^2, \quad \forall \mathbf{m} \in \R^2
\text{ such that } |\mathbf{m}|<\eta_4,
\end{equation}
which, together with \eqref{estimate-dottildeV}, proves the existence of $C>0$ such that, at least if $\mathbf{m}^0\in \R^2$ is small enough, the solution to \eqref{dotmathbfm} satisfies
\begin{equation}\label{estimate-m(t)}
|\mathbf{m}(t)|\leq \frac{C |\mathbf{m}^0|}{\sqrt{1+t|\mathbf{m}^0|^2}}, \quad \forall t\geq 0.
\end{equation}

It only remains to prove \eqref{not-degenerate}. From the Appendix, 
one gets that $c'(0)\approx0.0118 \neq0$, then \eqref{not-degenerate}
holds if $m_1=0$. Let us now deal with the case $m_1\neq 0$. Dividing
both the polynomials on the two equations on the left hand side of \eqref{not-degenerate}  by $m_1^{2}$, then the
two resulting polynomials have a common  root if and only if  their resultant is zero. This resultant
is the determinant of the Sylvester matrix $S$:
\begin{align}
S:= \begin{pmatrix}
c'(0) & b'(0)& a'(0)&0\\
0& c'(0)& b'(0)& a'(0)\\
-b'(0) & -2(a'(0)-c'(0))& b'(0)&0\\
0& -b'(0)& -2(a'(0)-c'(0))&b'(0)
\end{pmatrix}.
\end{align}
Straightforward computations show that
\begin{align}
{\text{det}}(S)=&a'(0)^3[b'(0)+4c'(0)]+a'(0)^2[-2b'(0)^2+b'(0)c'(0)-8c'(0)^2]\nonumber\\
&+a'(0)[5b'(0)^2c'(0)+4c'(0)^3]-b'(0)^2c'(0)^2-b'(0)^4.
\label{eq-detS}
\end{align}
From \eqref{eq-detS} and the Appendix
 (see in particular \eqref{a21ex}, \eqref{a23ex} and \eqref{a22ex}), we have
\begin{align}
{\text{det}}(S)&\approx-0.0197\neq0.
\end{align}
Hence, the
two resulting polynomials do not have a common  root. Thus, \eqref{not-degenerate} is proved.

\begin{rem}
It follows from our proof of Theorem~\ref{main result} that the decay rate stated in \eqref{estimatey(t)infinity} is optimal in the following sense: there exists $\varepsilon>0$ such that, for every $y_0\in G$ such that $\Vert y_0\Vert_{L^2(0,L)}\leq \varepsilon$,
\begin{equation}\label{lower-bound}
\Vert y(t,\cdot)\Vert _{L^2(0,L)}\geq \frac{\varepsilon \Vert y_0\Vert _{L^2(0,L)}}{\sqrt{1+t\Vert y_0\Vert _{L^2(0,L)}^2}}.
\end{equation}
$\Big($For the  Lyapunov approach, let us point out that, decreasing if necessary $\eta_3>0$, one has,  for every
 $\mu\in (0,\eta_3)$,
\begin{equation}\label{lower-dottildeV}
\dot{\tilde V}\geq - \frac{1}{\eta_3} |\mathbf{m}|^4 , \quad \forall \mathbf{m} \in \R^2
\text{ such that } |\mathbf{m}|<\eta_3.\Big)
\end{equation}
\end{rem}


\section{Conclusion and future works}
\label{sec-conclusion}\

In this article, we have proved that for the
critical case of $L=2\pi\sqrt{7/3}$, $0\in L^2(0,L)$ is locally asymptotically stable for the KdV equation \eqref{original}.
First, we recalled that the equation has a two-dimensional local center manifold. Next, through a second order power series approximation at $0\in M$ of the function $g$ defining the local center manifold, we derived the local asymptotic stability of  $0\in L^2(0,L)$ on the local center manifold and obtained a polynomial decay rate for the solution to the KdV equation \eqref{original} on the center manifold.

Since  the KdV equation  \eqref{original}
also has other (periodic) steady
states than the origin (see Remark~\ref{rem_nonunique}), it remains an open and interesting problem to consider the (local) stability property of these steady states for the  KdV equation \eqref{original}. Furthermore, it remains to consider all the other critical cases with a two-dimensional (local) center manifold as well as
all the last remaining critical cases, i.e., when the equation has a  (local) center manifold with a dimension larger than 2.

\Appendix

\section{On the solution $a$, $b$ and $c$
to equations \eqref{a21}, \eqref{a22} and \eqref{a23}}

Set
\begin{align}
f_+(x):=a(x)+c(x),~f_-(x):=a(x)-c(x),\label{f+-}
\end{align}
and
\begin{equation}
\left\{
\begin{array}
[c]{l}
g_+(x):=\varphi_1(x)\varphi_{1}'(x)+\varphi_2(x)\varphi_{2}'(x)+\sqrt{3}c_1\varphi_1(x)-c_1\varphi_2(x),\\
g_-(x):=\varphi_1(x)\varphi_{1}'(x)-\varphi_2(x)\varphi_{2}'(x)-\sqrt{3}c_1\varphi_1(x)-c_1\varphi_2(x),\\
g(x):=\varphi_{1}(x)\varphi_{2}'(x)+\varphi_{1}'(x)\varphi_{2}(x)+c_1\varphi_1(x)-\sqrt{3}c_1\varphi_2(x).
\end{array}
\right.\label{g+-0}
\end{equation}
First, adding each equation of \eqref{a23} to the corresponding
equation
of \eqref{a21}, we have
the following ODE equation for $f_+(x)$:\begin{equation}
\left\{
\begin{array}
[c]{l}
f_+'''(x)+f_+'(x)+g_+(x)=0,\\
f_+(0)=f_+(L)=0,~f_+'(L)=0.
\end{array}
\right.  \label{a24}
\end{equation}
Second, subtracting each equation of \eqref{a23} from the corresponding equation
of \eqref{a21}, we obtain
\begin{equation}
\left\{
\begin{array}
[c]{l}
2qb(x)+f_-'(x)+f_-'''(x)+g_-(x)=0,\\
f_-(0)=f_-(L)=0,~f_-'(L)=0,
\end{array}
\right.  \label{a25}
\end{equation}
which gives
\begin{align}
b(x)=-\frac{1}{2q}(f_-'(x)+f_-'''(x)+g_-(x)).\label{a22f}
\end{align}
Substitute \eqref{a22f} into \eqref{a22}, then
the following ODE equation for $f_-(x)$ is obtained:
\begin{equation}
\left\{
\begin{array}
[c]{l}
 f_-^{(6)}(x)+2f_-^{(4)}(x)+f_-''(x)+4q^2f_-(x)+g_-'(x)+g_-'''(x)-2qg(x)=0,
\\
f_-(0)=f_-(L)=~f_-'(L)=f_-'''(L)=0,\\
 f_-'(0)+f_-'''(0)=0,~f_-''(L)+f_-^{(4)}(L)=0,
\end{array}
\right.  \label{a26}
\end{equation}
where the second and third lines are borrowed from \eqref{a25}, and the last
three lines are obtained from
\eqref{a22f} and the boundary conditions of \eqref{eqn_varphi1}, \eqref{eqn_varphi2},
\eqref{a22}, \eqref{g+-0}, and
\eqref{a25}.

By employing the method of undetermined coefficients, the (unique) solution
to the  the nonhomogeneous ODE equation \eqref{a24}
is
\begin{align}
f_{+}(x) &=\sum\limits_{l=1}^{3} C_{+l}f_{+l}(x)\!+\!c_{+11}\!\cos\left(\frac{1}{\sqrt{21}}x\!\right)\!\!+c_{+12}\sin\left(\frac{1}{\sqrt{21}}x\right)\!+\!c_{+21}\cos\left(\frac{3}{\sqrt{21}}x\right)\nonumber\\
&~~~+c_{+31}\cos\left(\frac{4}{\sqrt{21}}x\right)+c_{+32}\sin\left(\frac{4}{\sqrt{21}}x\right)+c_{+41}\cos\left(\frac{5}{\sqrt{21}}x\right)\nonumber\\
&~~~+c_{+42}\sin\left(\frac{5}{\sqrt{21}}x\right)+c_{+51}\cos\left(\frac{6}{\sqrt{21}}x\right)+c_{+61}\cos\left(\frac{9}{\sqrt{21}}x\right),
\label{gen_sol_10}
\end{align}
where the fundamental solutions $f_{+l}(x), ~l=1, 2, 3$ are
\begin{align}
f_{+1}(x)=1, ~f_{+2}(x)=\cos(x), ~f_{+3}(x)=\sin(x),\label{fplus123}
\end{align}
and the constants are 
\begin{align}
&c_{+11}=\frac{3c_1\Theta}{(\frac{1}{\sqrt{21}})-(\frac{1}{\sqrt{21}})^3},
~~c_{+12}=\frac{-3\sqrt{3}c_1\Theta}{-(\frac{1}{\sqrt{21}})+(\frac{1}{\sqrt{21}})^3},
~d_{21}=\frac{\Theta^2\frac{18}{\sqrt{21}}}{(\frac{1}{\sqrt{21}})-(\frac{1}{\sqrt{21}})^3},\\
&c_{+31}=\frac{-2c_1\Theta}{(\frac{1}{\sqrt{21}})-(\frac{1}{\sqrt{21}})^3},~d_{32}=\frac{2\sqrt{3}c_1\Theta}{-(\frac{1}{\sqrt{21}})+(\frac{1}{\sqrt{21}})^3},~d_{41}=\frac{c_1\Theta}{(\frac{1}{\sqrt{21}})-(\frac{1}{\sqrt{21}})^3},\\
&c_{+42}=\frac{\sqrt{3}c_1\Theta}{-(\frac{1}{\sqrt{21}})+(\frac{1}{\sqrt{21}})^3},~d_{51}=\frac{\Theta^2\frac{18}{\sqrt{21}}}{(\frac{1}{\sqrt{21}})-(\frac{1}{\sqrt{21}})^3},~d_{61}=\frac{\Theta^2\frac{-18}{\sqrt{21}}}{(\frac{1}{\sqrt{21}})-(\frac{1}{\sqrt{21}})^3},
\end{align}
and
\begin{align}
C_{+l}=\frac{{\text{det}}(A_{+l})}{{\text{det}}(A_{+})},~l=1, 2,  3.
\end{align}
Here,
\begin{align}
&A_{+}=\begin{pmatrix}
f_{+1}(0) & f_{+2}(0) & f_{+3}(0)\\
f_{+1}(L) & f_{+2}(L) & f_{+3}(L)\\
f_{+1}'(L) & f_{+2}'(L) & f_{+3}'(L)
\end{pmatrix},
\end{align}
and each $ A_{+l}$  is the matrix formed by replacing the $l$-th column of $A_+$ with a column vector $-b_{+}$, where
\begin{align}
b_{+}=\begin{pmatrix}
b_{+1} & b_{+2} & b_{+3}
\end{pmatrix}\tr ,
\end{align}
and
\begin{align}
&b_{+1}=c_{+11}+c_{+21}+c_{+31}+c_{+41}+c_{+51}+c_{+61},\\
&b_{+2}=c_{+11}\cos\left(\frac{1}{\sqrt{21}}L\right)+c_{+12}\sin\left(\frac{1}{\sqrt{21}}L\right)
+c_{+21}\cos\left(\frac{3}{\sqrt{21}}L\right)
\nonumber\\
&~~~~~~~+c_{+31}\cos\left(\frac{4}{\sqrt{21}}L\right)
+c_{+32}\sin\left(\frac{4}{\sqrt{21}}L\right)
+c_{+41}\cos\left(\frac{5}{\sqrt{21}}L\right)\nonumber\\
&~~~~~~~+c_{+42}\sin\left(\frac{5}{\sqrt{21}}L\right)
+c_{+51}\cos\left(\frac{6}{\sqrt{21}}L\right)
+c_{+61}\cos\left(\frac{9}{\sqrt{21}}L\right),\\
&b_{+3}\!=\!-\frac{1}{\sqrt{21}}c_{+11}\!\sin\left(\!\frac{1}{\sqrt{21}}L\!\right)\!\!
+\!\!\frac{1}{\sqrt{21}}c_{+12}\!\cos\left(\!\frac{1}{\sqrt{21}}L\!\right)\!\!
-\!\!\frac{3}{\sqrt{21}}c_{+21}\!\sin\left(\frac{3}{\sqrt{21}}L\!\right)\nonumber\\
&~~~~~~~~\!-\!\frac{4}{\sqrt{21}}c_{+31}\!\sin\left(\!\frac{4}{\sqrt{21}}L\!\right)\!\!
+\!\frac{4}{\sqrt{21}}c_{+32}\!\cos\left(\!\frac{4}{\sqrt{21}}L\!\right)\!\!
-\!\!\frac{5}{\sqrt{21}}c_{+41}\!\sin\left(\frac{5}{\sqrt{21}}L\!\right)\nonumber\\
&~~~~~~~~\!+\!\frac{5}{\sqrt{21}}c_{+42}\!\cos\left(\!\frac{5}{\sqrt{21}}L\!\right)\!\!
-\!\!\frac{6}{\sqrt{21}}c_{+51}\!\sin\left(\!\frac{6}{\sqrt{21}}L\!\right)\!\!
-\!\!\frac{9}{\sqrt{21}}c_{+61}\!\sin\left(\frac{9}{\sqrt{21}}L\!\right).
\end{align}

Similarly, by employing the method of undetermined coefficients, the (unique) solution to the nonhomogeneous ODE system \eqref{a26} is
\begin{align}
f_{-}(x) =&\sum\limits_{l=1}^{6}C_{-l}f_{-l}(x)+c_{-11}\cos\left(\frac{1}{\sqrt{21}}x\right)+c_{-12}\sin\left(\frac{1}{\sqrt{21}}x\right)+c_{-21}\cos\left(\frac{2}{\sqrt{21}}x\right)\nonumber\\
&+c_{-31}\cos\left(\frac{4}{\sqrt{21}}x\right)+c_{-32}\sin\left(\frac{4}{\sqrt{21}}x\right)+c_{-41}\cos\left(\frac{5}{\sqrt{21}}x\right)\nonumber\\
&+c_{-42}\sin\left(\frac{5}{\sqrt{21}}x\right)+c_{-51}\cos\left(\frac{8}{\sqrt{21}}x\right)+c_{-61}\cos\left(\frac{10}{\sqrt{21}}x\right),
\label{gen_sol_1}
\end{align}
where the fundamental solutions
$f_{-l}, l=\overline{1,6}$ are \begin{equation}
\left\{
\begin{array}
[c]{l}
f_{-1}(x)=e^{\alpha _{1}x}\cos \left(
\beta _{1}x\right)
,~~~f_{-2}(x)=e^{\alpha _{1}x}\sin \left( \beta _{1}x\right),\\
f_{-3}(x)=e^{-\alpha
_{1}x}\cos \left( \beta _{1}x\right),~f_{-4}(x)=e^{-\alpha _{1}x}\sin \left(
\beta _{1}x\right),\\
f_{-5}(x)=\cos \left( \beta _{2}x\right),~~~~~~~~~f_{-6}(x)=\sin \left( \beta _{2}x\right),\label{fmin56}
\end{array}
\right.
\end{equation}
with
\begin{eqnarray}
\alpha _{1} &=&\frac{\left(
20+\sqrt{57}\right) ^{\frac{1}{3}}-7\left(
20+\sqrt{57}\right) ^{-\frac{1}{3}} }{2\sqrt{7}}, \\
\beta _{1} &=&\frac{\left(
20+\sqrt{57}\right) ^{\frac{1}{3}}+7\left(
20+\sqrt{57}\right) ^{-\frac{1}{3}} }{2\sqrt{21}}, \\
\beta _{2} &=&
\frac{\left(
20+\sqrt{57}\right) ^{\frac{1}{3}}+7\left(
20+\sqrt{57}\right) ^{-\frac{1}{3}} }{\sqrt{21}},
\end{eqnarray}
the constants are
\begin{align}
&c_{-11}=\frac{-3\Theta^2\frac{40}{21^2}
+4q\Theta^2\frac{2}{\sqrt{21}}+9qc_1\Theta}{(\frac{1}{\sqrt{21}})^6-2(\frac{1}{\sqrt{21}})^4+(\frac{1}{\sqrt{21}})^2-4q^2},\\
&c_{-12}=\frac{-9\sqrt{3}qc_1\Theta}{(\frac{1}{\sqrt{21}})^6-2(\frac{1}{\sqrt{21}})^4+(\frac{1}{\sqrt{21}})^2-4q^2},\\
&c_{-21}=\frac{3\Theta^2\frac{18}{21^2}-4q\Theta^2\frac{3^2}{\sqrt{21}}}{(\frac{2}{\sqrt{21}})^6-
2(\frac{2}{\sqrt{21}})^4+(\frac{2}{\sqrt{21}})^2-4q^2},\\
&c_{-31}=\frac{3\Theta^2\frac{240}{21^2}-4q\Theta^2\frac{12}{\sqrt{21}}
-6qc_1\Theta}{(\frac{4}{\sqrt{21}})^6-2(\frac{4}{\sqrt{21}})^4+(\frac{4}{\sqrt{21}})^2-4q^2},\\
&c_{-32}=\frac{6\sqrt{3}qc_1\Theta}{(\frac{4}{\sqrt{21}})^6-2(\frac{4}{\sqrt{21}})^4+(\frac{4}{\sqrt{21}})^2-4q^2},
\\
&c_{-41}=\frac{-3\Theta^2\frac{600}{21^2}+4q\Theta^2\frac{30}{\sqrt{21}}
-3qc_1\Theta}{(\frac{5}{\sqrt{21}})^6-2(\frac{5}{\sqrt{21}})^4+(\frac{5}{\sqrt{21}})^2-4q^2},\\
&c_{-42}=\frac{-3\sqrt{3}qc_1\Theta}{(\frac{5}{\sqrt{21}})^6-2(\frac{5}{\sqrt{21}})^4
+(\frac{5}{\sqrt{21}})^2-4q^2},\\
&c_{-51}=\frac{3\Theta^2\frac{2048}{21^2}-4q\Theta^2\frac{16}{\sqrt{21}}}{(\frac{8}{\sqrt{21}})^6-2(\frac{8}{\sqrt{21}})^4
+(\frac{8}{\sqrt{21}})^2-4q^2},\\
&c_{-61}=\frac{3\Theta^2\frac{1250}{21^2}+4q\Theta^2\frac{5}{\sqrt{21}}}{(\frac{10}{\sqrt{21}})^6
-2(\frac{10}{\sqrt{21}})^4+(\frac{10}{\sqrt{21}})^2-4q^2},
\end{align}
and
\begin{align}
C_{-l}=\frac{{\text{det}}(A_{-l})}{{\text{det}}(A_{-})},~l=\overline{1,6}.
\end{align}
Here, the matrix
\begin{align}
&A_{-}=\begin{pmatrix}
\alpha_{-1} & \alpha_{-2} & \alpha_{-3} & \alpha_{-4} & \alpha_{-5} & \alpha_{-6}
\end{pmatrix},
\end{align}
where
\begin{align}
\alpha_{-l}=\begin{pmatrix}
 f_{-l}(0)\\
 f_{-l}(L)\\
 f_{-l}'(L)\\
 f_{-l}'(0)+f_{-l}'''(0)\\
 f_{-l}'''(L)\\
 f_{-l}''(L)+f_{-l}^{(4)}(L)
\end{pmatrix}, ~l=\overline{1,6},
\end{align}
and each $ A_{-l}$  is the matrix formed by replacing the $l$-th column of
$A_-$ with a column vector $-b_{-}$, where
\begin{align}
b_{-}=\begin{pmatrix}
b_{-1} & b_{-2} & b_{-3} & b_{-4} & b_{-5} & b_{-6}
\end{pmatrix}\tr .
\end{align}
Here,
\begin{align}
&b_{-1}=c_{-11}+c_{-21}+c_{-31}+c_{-41}+c_{-51}+c_{-61},
\end{align}
\begin{eqnarray}
b_{-2}&=&c_{-11}\cos\left(\frac{1}{\sqrt{21}}L\right)
+c_{-12}\sin\left(\frac{1}{\sqrt{21}}L\right)
+c_{-21}\cos\left(\frac{2}{\sqrt{21}}L\right)
\nonumber\\
&&+c_{-31}\cos\left(\frac{4}{\sqrt{21}}L\right)+c_{-32}\sin\left(\frac{4}{\sqrt{21}}L\right)
+c_{-41}\cos\left(\frac{5}{\sqrt{21}}L\right)\nonumber\\
&&+c_{-42}\sin\left(\frac{5}{\sqrt{21}}L\right)+c_{-51}\cos\left(\frac{8}{\sqrt{21}}L\right)
+c_{-61}\cos\left(\frac{10}{\sqrt{21}}L\right),
\end{eqnarray}
\begin{eqnarray}
b_{-3}&=&-\frac{1}{\sqrt{21}}c_{-11}\sin\left(\frac{1}{\sqrt{21}}L\right)
+\frac{1}{\sqrt{21}}c_{-12}\cos\left(\frac{1}{\sqrt{21}}L\right)\nonumber\\
&&-\frac{2}{\sqrt{21}}c_{-21}\sin\left(\frac{2}{\sqrt{21}}L\right)\nonumber\\
&&-\frac{4}{\sqrt{21}}c_{-31}\sin\left(\frac{4}{\sqrt{21}}L\right)
+\frac{4}{\sqrt{21}}c_{-32}\cos\left(\frac{4}{\sqrt{21}}L\right)\nonumber\\
&&-\frac{5}{\sqrt{21}}c_{-41}\sin\left(\frac{5}{\sqrt{21}}L\right)
+\frac{5}{\sqrt{21}}c_{-42}\cos\left(\frac{5}{\sqrt{21}}L\right)\nonumber\\
&&-\frac{8}{\sqrt{21}}c_{-51}\sin\left(\frac{8}{\sqrt{21}}L\right)
-\frac{10}{\sqrt{21}}c_{-61}\sin\left(\frac{10}{\sqrt{21}}L\right),
\end{eqnarray}
\begin{eqnarray}
b_{-4}&=&\frac{20}{21\sqrt{21}}c_{-12}\cos\left(\frac{1}{\sqrt{21}}L\right)+\frac{20}{21\sqrt{21}}c_{-32}\cos\left(\frac{4}{\sqrt{21}}L\right)\nonumber\\
&&-\frac{20}{\sqrt{21}}c_{-42}\cos\left(\frac{5}{\sqrt{21}}L\right),
\end{eqnarray}
\begin{eqnarray}
b_{-5}&=&-\frac{20}{21\sqrt{21}}c_{-11}\sin\left(\frac{1}{\sqrt{21}}L\right)
+\frac{20}{21\sqrt{21}}c_{-12}\cos\left(\frac{1}{\sqrt{21}}L\right)\nonumber\\
&&-\frac{34}{21\sqrt{21}}c_{-21}\sin\left(\frac{2}{\sqrt{21}}L\right)\nonumber\\
&&-\frac{20}{21\sqrt{21}}c_{-31}\sin\left(\frac{4}{\sqrt{21}}L\right)
+\frac{20}{21\sqrt{21}}c_{-32}\cos\left(\frac{4}{\sqrt{21}}L\right)\nonumber\\
&&+\frac{20}{21\sqrt{21}}c_{-41}\sin\left(\frac{5}{\sqrt{21}}L\right)
-\frac{20}{21\sqrt{21}}c_{-42}\cos\left(\frac{5}{\sqrt{21}}L\right)\nonumber\\
&&+\frac{344}{21\sqrt{21}}c_{-51}\sin\left(\frac{8}{\sqrt{21}}L\right)
+\frac{790}{21\sqrt{21}}c_{-61}\sin\left(\frac{10}{\sqrt{21}}L\right),
\end{eqnarray}
\begin{eqnarray}
b_{-6}&=&-\frac{20}{21^2}c_{-11}\cos\left(\frac{1}{\sqrt{21}}L\right)
-\frac{20}{21^2}c_{-12}\sin\left(\frac{1}{\sqrt{21}}L\right)
\nonumber\\
&&-\frac{68}{21^2}c_{-21}\cos\left(\frac{2}{\sqrt{21}}L\right)
\nonumber\\
&&-\frac{80}{21^2}c_{-31}\cos\left(\frac{4}{\sqrt{21}}L\right)
-\frac{80}{21^2}c_{-32}\sin\left(\frac{4}{\sqrt{21}}L\right)\nonumber\\
&&+\frac{100}{21^2}c_{-41}\cos\left(\frac{5}{\sqrt{21}}L\right)
+\frac{100}{21^2}c_{-42}\sin\left(\frac{5}{\sqrt{21}}L\right)\nonumber\\
&&+\frac{2752}{21^2}c_{-51}\cos\left(\frac{8}{\sqrt{21}}L\right)
+\frac{7900}{21^2}c_{-61}\cos\left(\frac{10}{\sqrt{21}}L\right).
\end{eqnarray}

Therefore, we derive from \eqref{f+-} that
\begin{eqnarray}
a(x)&=&\displaystyle
\frac{1}{2}(f_+(x)+f_-(x))\nonumber\\
&=&\displaystyle \frac{1}{2}\bigg[ \sum\limits_{l=1}^{3} C_{+l}f_{+l}(x)+\sum\limits_{l=1}^{6}C_{-l}f_{-l}(x)\nonumber\\
&&\displaystyle +(c_{+11}+c_{-11})\cos\left(\frac{1}{\sqrt{21}}x\right)+(c_{+12}+c_{-12})\sin\left(\frac{1}{\sqrt{21}}x\right)\nonumber\\
&&\displaystyle +c_{-21}\cos\left(\frac{2}{\sqrt{21}}x\right)+c_{+21}\cos\left(\frac{3}{\sqrt{21}}x\right)\nonumber\\
&&\displaystyle +(c_{+31}+c_{-31})\cos\left(\frac{4}{\sqrt{21}}x\right)+(c_{+32}+c_{-32})\sin\left(\frac{4}{\sqrt{21}}x\right)\nonumber\\
&&\displaystyle +(c_{+41}+c_{-41})\cos\left(\frac{5}{\sqrt{21}}x\right)+(c_{+42}+c_{-42})\sin\left(\frac{5}{\sqrt{21}}x\right)\nonumber\\
&&\displaystyle +c_{+51}\cos\left(\frac{6}{\sqrt{21}}x\right)+c_{-51}\cos\left(\frac{8}{\sqrt{21}}x\right)\nonumber\\
&&\displaystyle +c_{+61}\cos\left(\frac{9}{\sqrt{21}}x\right)+c_{-61}\cos\left(\frac{10}{\sqrt{21}}x\right)\bigg],\label{a21ex}
\end{eqnarray}
and
\begin{eqnarray}
c(x)&=& \displaystyle \frac{1}{2}(f_+(x)-f_-(x))\nonumber\\
&=&\displaystyle \frac{1}{2}\bigg[ \sum\limits_{l=1}^{3} C_{+l}f_{+l}(x)-\sum\limits_{l=1}^{6}C_{-l}f_{-l}(x)\nonumber\\
&&\displaystyle +(c_{+11}-c_{-11})\cos\left(\frac{1}{\sqrt{21}}x\right)+(c_{+12}-c_{-12})\sin\left(\frac{1}{\sqrt{21}}x\right)\nonumber\\
&&\displaystyle -c_{-21}\cos\left(\frac{2}{\sqrt{21}}x\right)+c_{+21}\cos\left(\frac{3}{\sqrt{21}}x\right)\nonumber\\
&&\displaystyle +(c_{+31}-c_{-31})\cos\left(\frac{4}{\sqrt{21}}x\right)+(c_{+32}-c_{-32})\sin\left(\frac{4}{\sqrt{21}}x\right)\nonumber\\
&&\displaystyle +(c_{+41}-c_{-41})\cos\left(\frac{5}{\sqrt{21}}x\right)+(c_{+42}-c_{-42})\sin\left(\frac{5}{\sqrt{21}}x\right)\nonumber\\
&&\displaystyle +c_{+51}\cos\left(\frac{6}{\sqrt{21}}x\right)-c_{-51}\cos\left(\frac{8}{\sqrt{21}}x\right)\nonumber\\
&&\displaystyle +c_{+61}\cos\left(\frac{9}{\sqrt{21}}x\right)-c_{-61}\cos\left(\frac{10}{\sqrt{21}}x\right)\bigg].\label{a23ex}
\end{eqnarray}
From \eqref{a22f}, we obtain
\begin{eqnarray}
b(x)&=&-\frac{1}{2q}(f_-'(x)+f_-'''(x)+g_-(x))\nonumber\\
&=&-\frac{1}{2q}\bigg[ \sum\limits_{l=1}^{6}C_{-l}f_{-l}'(x)+\sum_{l=1}^{6}C_{-l}f_{-l}'''(x)\nonumber\\
&&-\left(\frac{20}{21\sqrt{21}}c_{-11}
+\frac{2}{\sqrt{21}}\Theta^2+3c_1\Theta\right)\sin\left(\frac{1}{\sqrt{21}}x\right)\nonumber\\
&&+\left(\frac{20}{21\sqrt{21}}c_{-12}
+3\sqrt{3}c_1\Theta\right)\cos\left(\frac{1}{\sqrt{21}}x\right)\nonumber\\
&&-\left(\frac{34}{21\sqrt{21}}c_{-21}+\frac{9}{\sqrt{21}}\Theta^2\right)\sin\left(\frac{2}{\sqrt{21}}x\right)\nonumber\\
&&-\left(\frac{20}{21\sqrt{21}}c_{-31}+2c_1\Theta\right)\sin\left(\frac{4}{\sqrt{21}}x\right)\nonumber\\
&&+\left(\frac{20}{21\sqrt{21}}c_{-32}-2\sqrt{3}c_1\Theta\right)\cos\left(\frac{4}{\sqrt{21}}x\right)\nonumber\\
&&+\left(\frac{20}{21\sqrt{21}}c_{-41}+\frac{30}{\sqrt{21}}\Theta^2+c_1\Theta\right)\sin\left(\frac{5}{\sqrt{21}}x\right)\nonumber\\
&&-\left(\frac{20}{21\sqrt{21}}c_{-42}+\sqrt{3}c_1\Theta\right)\cos\left(\frac{5}{\sqrt{21}}x\right)
\nonumber\\
&&-\frac{12}{\sqrt{21}}\Theta^2\sin\left(\frac{6}{\sqrt{21}}x\right)
+\left(\frac{8\times43}{21\sqrt{21}}c_{-51}
-\frac{16}{\sqrt{21}}\Theta^2\right)\sin\left(\frac{8}{\sqrt{21}}x\right)\nonumber\\
&&+\left(\frac{790}{\sqrt{21}}c_{-61}-\frac{5}{\sqrt{21}}\Theta^2\right)\sin\left(\frac{10}{\sqrt{21}}x\right)\bigg].\label{a22ex}
\end{eqnarray}

\textbf{Acknowledgment.} We would like to thank Shengquan Xiang for useful
comments on a preliminary version of this article.

\bibliographystyle{plain}
\bibliography{TSXKdVnonlinear}

\end{document}